\documentclass[11pt,twoside]{article}
\usepackage{amsfonts}
\usepackage{graphicx,color}
\usepackage{amsthm}
\usepackage{amssymb,amsmath}

\def\ds{\displaystyle}

\newtheorem{theorem}{Theorem}[section]
\newtheorem{lemma}[theorem]{Lemma}

\date{\today}

\theoremstyle{definition}

\newtheorem{remark}[theorem]{Remark}

\graphicspath{{./Figs/}}
\title{Upper Energy Bounds for Spherical Designs of Relatively Small Cardinalities}

\begin{document}

\author{Peter Boyvalenkov\thanks{The research of the first two authors was supported, in part, by Bulgarian NSF contract DN02/2-2016.} \\
Institute of Mathematics and Informatics, Bulgarian Academy of Sciences, \\
8 G Bonchev Str., 1113  Sofia, Bulgaria; peter@math.bas.bg \\[6pt]
Konstantin Delchev \\
Institute of Mathematics and Informatics, Bulgarian Academy of Sciences, \\
8 G Bonchev Str., 1113  Sofia, Bulgaria; math\_k\_delchev@yahoo.com \\[6pt]
Matthieu Jourdain \\
\'{E}cole de Saint-Cyr Co\"{e}tquidan, D\'{e}partement des sciences de l'ing\'{e}nieur, \\
56381 GUER Cedex, France; matthieu.jourdain@st-cyr.terre-net.defense.gouv.fr}

\date{}

\maketitle

\begin{abstract} \noindent
We derive upper bounds for the potential energy
of spherical designs of cardinality close to the Delsarte-Goethals-Seidel bound.
These bounds are obtained by linear programming with the use of the Hermite interpolating polynomial of the potential function in suitable nodes.
Numerical computations show that the results are quite close to certain lower energy bounds confirming that spherical designs are, in a sense, energy efficient.
\end{abstract}



\section{Introduction}
\label{intro}
Spherical designs were introduced in 1977 by Delsarte, Goethals and Seidel in the seminal paper \cite{DGS}.
Let $\mathbb{S}^{n-1}$ be the unit sphere in $\mathbb{R}^n$. A
spherical $\tau$-design $C \subset \mathbb{S}^{n-1}$ as a nonempty finite set $C \subset \mathbb{S}^{n-1}$
\begin{eqnarray*}
\frac{1}{\mu(\mathbb{S}^{n-1})} \int_{\mathbb{S}^{n-1}} f(x) d\mu(x)= \frac{1}{|C|} \sum_{x \in C} f(x),
\end{eqnarray*}
where $\mu(x)$ is the surface area measure, holds for all polynomials $f(x) = f(x_1,x_2,\ldots,x_n)$ of total degree at most $\tau$.
The maximal number $\tau = \tau(C)$ such that $C$ is a spherical $\tau$-design is called the {\em strength} of $C$.

Spherical designs have broad applications due to two major properties -- their connection to numeric integration on the sphere and their tendency to lead to highly symmetric geometric configurations. In fact many optimal configurations,  such as the regular polytopes, the Korkine-Zolotarev lattice and the Leech lattice, which arise naturally from various problems are connected to spherical designs with good parameters. Futhermore, a special class of spherical designs,
so called sharp configurations, were shown to be universally optimal in the sense of Cohn-Kumar \cite{CK}.
Classical applications of spherical designs include connections with Waring problem \cite{Rez},
isometric embeddings between Banach spaces \cite{LW}, and Chebyshev-type quadrature formulas
\cite{Kuj1,Kuj2}, which in turn provides efficient tools for numerical integration on the sphere \cite{RB,BV}.

An important subclass of spherical designs, the designs on $\mathbb{S}^2$ with icosahedral symmetry, geodesic grids, were introduced  in 1968 by Sadourny, Arakawa and Mint \cite{SAM}, and find increasing importance in various fields such as global ocean modeling \cite{TPHJJM} and environmental monitoring \cite{WKO}. They have richer symmetry than the standard longitude-latitude geographic grid, faces with mostly equal area, can be easily extended with iterative inclusion of edge midpoints and have the integration properties of 5-designs making them suitable base for geographic indexing.
Another similar extendable structure, but based on cubic symmetry, is the quadrilateralized spherical cube, which is also a 3-design. It was introduced in 1975 and famously implemented in the Cosmic Background Explorer project. Similar approach, combined with space-filling curves is also used by the Google S2 Geometry library. Other current and suggested applications of designs on $\mathbb{S}^2$ include computer graphics, 3D scanning and design of LEO and GSO satellite constellations.

For a given (extended real-valued) function $h:[-1,1] \to [0,+\infty]$, the {\em $h$-energy}
(or the potential energy) of a spherical $\tau$-design $C \subset \mathbb{S}^{n-1}$ is defined by
\begin{equation}\label{Energy}
E(n,C;h):=\sum_{x, y \in C, x \neq y} h(\langle x,y \rangle),
\end{equation}
where $\langle x,y \rangle$ denotes the inner product of $x$ and $y$. In what follows we assume that $h$ is absolutely monotone, i.e. $h^{(k)}(t) \geq 0$ for every $k \geq 0$ and $t \in [-1,1)$. A common example is the Riesz $s$-potential $h(t)=1/(2(1-t))^s$.

Another important feature of the spherical designs was established in 2006 by Cohn and Kumar \cite{CK}.
It was proved in \cite{CK} that the sharp configurations (spherical designs of strength $\tau=2m-1$ and have exactly $m$ distinct distances between their points; see also the survey papers \cite{BB,BBTZ})
are universally optimal; i.e., for every $h$ they possess the
minimum possible $h$-energy among all spherical codes on $\mathbb{S}^{n-1}$ of the same cardinality.
The only addition to this list is the 600-cell on $\mathbb{S}^3$ which is a spherical 11-design.

Bounds for the energy of spherical designs on $\mathbb{S}^2$ were obtained in \cite{H,HL} for particular $h$.
The discrete Riesz $s$-energy of sequences of well separated $\tau$-designs was
investigated in \cite{GS}. General lower and upper bounds on energy of designs of fixed dimension, strength and cardinality
were obtained by Boyvalenkov-Dragnev-Hardin-Saff-Stoyanova \cite{BDHSS-designs} (see also \cite{BDHSS-codes}).

In this paper we address the general problem for finding upper bounds, i.e., to estimate from above the
quantity
\begin{equation}
\label{UE}
\mathcal{U}(n,M,\tau;h):=\sup \{E(n,C;h):|C|=M, \ C \subset \mathbb{S}^{n-1} \mbox{ is a $\tau$-design}\},
\end{equation}
the maximum possible $h$-energy of a spherical $\tau$-design of $M$ points on $\mathbb{S}^{n-1}$.
We use a linear programming approach (sometimes called Delsarte-Yudin method) with a Hermite interpolation polynomial of the potential function. Our main result suggests (or confirms) that the spherical designs are,so to say, energy effective. This means that all
designs on $\mathbb{S}^{n-1}$ of relatively small (fixed) cardinalities have their $h$-energy in very thin range. Indeed,
our upper bounds are very close to the recently obtained universal lower bound \cite{BDHSS-designs,BDHSS-codes}. As in
\cite{BDHSS-designs,BDHSS-codes,CK} our results are valid for all absolutely monotone functions $h$.

In Section 2 we present notations and results needed for the rest of the paper. In particular, the linear programming
technique is described in the context of the energy bounds. Section 3 is devoted to our new bound. We utilize the linear programming by Hermite interpolation to the potential function at suitable nodes. Two representations of our bounds are shown to connect our results to certain lower bounds via certain parameters introduced by Levenshtein \cite{Lev92}.
In Section 4 we present explicit bounds for $\tau=2$ and numerical examples for $\tau=4$.

\section{Preliminaries}
\subsection{Gegenbauer polynomials}

For fixed dimension $n$, the normalized Gegenbauer polynomials are
defined by  $P_0^{(n)}(t):=1$, $P_1^{(n)}(t):=t$ and the three-term recurrence relation
\[ (i+n-2)\, P_{i+1}^{(n)}(t)=(2i+n-2)\, t\, P_i^{(n)}(t)-i\, P_{i-1}^{(n)}(t)
                \mbox{ for } i \geq 1. \]
We have $P_i^{(n)}(t)=P_i^{(\alpha,\beta)}(t)/P_i^{(\alpha,\beta)}(1)$, $\alpha=\beta=(n-3)/2$,
where $P_i^{(\alpha,\beta)}(t)$ are the Jacobi polynomials in standard notation \cite{AS,Sze}.

If $f(t) \in \mathbb{R}[t]$ is a real polynomial of degree $r$, then
$f(t)$ can be uniquely expanded in terms of the Gegenbauer
polynomials as
\begin{equation}\label{GegenbauerExpantion}
f(t) = \sum_{i=0}^r f_iP_i^{(n)}(t).
\end{equation}

An important property of the Gegenbauer polynomials connects them to  harmonic analysis on $\mathbb{S}^{n-1}$
via the formula \cite{Koo}
\begin{equation}
\label{add-formula}
P_{i}^{(n)}(\langle x,y \rangle)=\frac{1}{r_i}\sum_{j=1}^{r_i} v_{ij}(x)v_{ij}(y)
\end{equation}
for any two points $x,y \in \mathbb{S}^{n-1}$. Here  $r_i=\mbox{dim(Harm($i$))}$ is
the dimension of the space of homogeneous harmonic polynomials of degree $i$ and $\{v_{ij}(x)\}_{j=1}^{r_i}$
is an orthonormal basis of that space. It is worth noting that the definition for spherical design can be also stated as $\sum_{x \in C}v(x)=0$ for every nonconstant homogeneous harmonic polynomial of degree at most $\tau$. Using \eqref{add-formula} one writes in two ways the
sum $\sum_{x,y \in C}f(\langle x,y \rangle)$ to reach the identity
\begin{equation}
\label{main-identity}
|C|f(1)+\sum_{x,y \in C, x \neq y}f(\langle x,y \rangle) = |C|^2 f_0 + \sum_{i=0}^{k}\frac{f_i}{r_i}\sum_{j=1}^{r_i}\left(\sum_{x \in C} v_{ij}(x)\right)^2
\end{equation}
which serves as a base for linear programming bounds for the cardinality and energy of spherical codes and designs
(cf. \cite{BDHSS-designs,BDHSS-codes,DGS,DL,Lev92,Lev}).

Furthermore, we will use another series of polynomials $P_{i}^{a,b}(t)$, $a,b \in \{0,1\}$,
the so-called adjacent polynomials, which are again Jacobi polynomials, now with parameters
\[ (\alpha,\beta)=(a+\frac{n-3}{2}, b+\frac{n-3}{2}) \]
normalized by $P_{i}^{a,b}(1)=1$. Note the special case $P^{0,0}_i(t)=P^{(n)}_i(t)$. We denote by
$t_{i,1}^{a,b}<\cdots<t_{i,i}^{a,b}$ the roots of the polynomial  $P_{i}^{a,b}(t)$.

\subsection{Delsarte-Goethals-Seidel bound and polynomials}

Denote
 $B(n,\tau):=\min\{|C|: C \subset \mathbb{S}^{n-1} \mbox{ is a spherical $\tau$-design}\}$.
Delsarte, Goethals, and Seidel \cite{DGS} obtained the following Fisher-type lower bound
\begin{equation}
\label{DGS-bound}
B(n,\tau) \geq D(n,\tau) := \left\{ \begin{array}{ll}
 \ds 2\binom{n+k-2}{n-1}, & \mbox{ if $\tau=2k-1$,} \\[12pt]
 \ds \binom{n+k-1}{n-1}+\binom{n+k-2}{n-1}, & \mbox{ if  $\tau=2k$}.
\end{array}
  \right.
\end{equation}

The bound \eqref{DGS-bound} was obtained by using in the next theorem the polynomials
\begin{eqnarray}
\label{DGS-poly}
     d_{\tau}(t) = \left\{
     \begin{array}{ll}
        (t+1)\left(P_{k-1}^{1,1}(t)\right)^2, & \mbox{if } \tau=2k-1 \\
        \left(P_k^{1,0}(t)\right)^2, & \mbox{if } \tau=2k
     \end{array} \right. .
\end{eqnarray}

\begin{theorem}
\label{LP-designs}
Let $n \geq 3$, $\tau \geq 1$, and $f(t)$ be a real valued polynomial such that:

{\rm (B1)} $f(t) \geq 0$ for  $ t \in [-1,1]$;

{\rm (B1)}  If $f(t)=\sum_{i=0}^{k} f_i P_{i}^{(n)}(t)$ then $f_i \leq 0$ for every $i > \tau$.

Then $B(n,\tau) \geq f(1)/f_0$.
\end{theorem}

The proof follows by applying (B1) to the left side and (B2) to the right hand side in \eqref{main-identity}.

Spherical $\tau$-designs which attain the bound \eqref{DGS-bound} are called tight. Tight $\tau$-designs can
exist for $\tau \in \{1,2,3,4,5,7,11\}$ only \cite{BD1,BD2}. Moreover, the inner products and distance distribution (therefore, the energy)
of the tight designs are well known. Thus we do not consider tight designs.

For many cardinalities close (but not equal) to $D(n,\tau)$ existence of spherical $\tau$-designs is still an open problem. On the other hand, existence of designs with asymptotically optimal cardinalities was proved by Bondarenko, Radchenko, and Viazovska \cite{BRV1,BRV2}.

\subsection{Levenshtein bounds on maximal cardinality of spherical codes of prescribed maximal inner product}

Denote
$$A(n,s)=\max\{|C|: C \subset \mathbb{S}^{n-1}, \langle x,y \rangle \leq s \mbox{ for all } x,y \in C, x \neq y\},$$
the maximal possible cardinality of a spherical code on $\mathbb{S}^{n-1}$ of prescribed maximal
inner product $s$. Levenshtein used linear programming techniques (see \cite{Lev}) to obtain the bound
\begin{equation}
\label{L_bnd}
 A(n,s) \leq
\left\{
\begin{array}{ll}
    L_{2k-1}(n,s) = {k+n-3 \choose k-1}
         \big[ \frac{2k+n-3}{n-1} -
          \frac{P_{k-1}^{(n)}(s)-P_k^{(n)}(s)}{(1-s)P_k^{(n)}(s)}
         \big], \\[12pt] &\mbox{if }s\in [t_{k-1}^{1,1},t_k^{1,0}] \\[12pt]
    L_{2k}(n,s) = {k+n-2 \choose k}
        \big[ \frac{2k+n-1}{n-1} -
           \frac{(1+s)( P_k^{(n)}(s)-P_{k+1}^{(n)}(s))}
    {(1-s)(P_k^{(n)}(s)+P_{k+1}^{(n)}(s))} \big], \\[12pt] &\mbox{if } s\in  [t_k^{1,0},t_{k}^{1,1}].\cr
   \end{array}\right.
\end{equation}

Important connections between the Delsarte-Goethals-Seidel bound (\ref{DGS-bound}) and the
Levenshtein bounds (\ref{L_bnd}) are given by the equalities
\begin{equation}\label{L-DGS1}
\begin{split}
L_{2k-2}(n,t_{k-1}^{1,1})&=
L_{2k-1}(n,t_{k-1}^{1,1}) = D(n,2k-1),\\
  L_{2k-1}(n,t_k^{1,0})&=
L_{2k}(n,t_k^{1,0}) = D(n,2k)
\end{split}
\end{equation}
at the boundaries of the intervals of the Levenshtein bounds.

\subsection{Levenshtein's $1/M$-quadrature rule}

The coefficient
\[ f_0=\int_{-1}^1 f(t)(1-t^2)^{\frac{n-3}{2}} d\, t \]
from the expansion \eqref{GegenbauerExpantion}
is crucial in the linear programming (see, for example, Theorems \ref{LP-designs} and \ref{general-upper}).
Let $\alpha_0<\alpha_1<\cdots<\alpha_{k-1}$ (resp. $\beta_1<\beta_2<\cdots<\beta_k$) be the roots of the equation
\[ P_k(t)P_{k-1}(s)=P_k(s)P_{k-1}(t), \]
where $P_i(t)=P_i^{1,0}(t)$ and $s=\alpha_{k-1}$ (resp. $P_i(t)=P_i^{1,1}(t)$ and $s=\beta_{k}$); $s$ will
be explained below. Finally, set $\beta_0=-1$.

Levenshtein \cite{Lev92} (see \cite[Section 5]{Lev} for comprehensive explanation) proved that the Gauss-Jacobi-type formula
 \begin{equation}
\label{defin_f0.1}
f_0=\left\{
     \begin{array}{ll}
    \ds    \frac{f(1)}{L_{2k-1}(n,s)}+ \sum_{i=0}^{k-1} \rho_i f(\alpha_i), & \mbox{if } \tau=2k-1 \\
    \ds     \frac{f(1)}{L_{2k}(n,s)}+ \sum_{i=0}^{k} \gamma_i f(\beta_i), & \mbox{if } \tau=2k
     \end{array} \right.
\end{equation}
($\rho_i, \gamma_i$ are positive weights) holds true for all polynomials $f$ of degree at most $\tau$.

It was observed in \cite{BBD}  that \eqref{defin_f0.1} can be formulated to
serve for investigation of the structure of spherical designs when $L_\tau(n,s)$ is replaced by
the cardinality $M$ of a putative spherical $\tau$-design $C \subset \mathbb{S}^{n-1}$. Then the
design's cardinality $M=L_\tau(n,s)$
comes as uniquely associated with the corresponding numbers:
\begin{equation}
\begin{array}{ll}
\alpha_0<\alpha_1<\cdots<\alpha_{k-1}=s, \ \rho_0,\rho_1,\ldots,\rho_{k-1}  , & \mbox{if } \tau=2k-1, \\[6pt]
-1=\beta_0<\beta_1<\cdots<\beta_{k}=s, \ \gamma_0,\gamma_1,\ldots,\gamma_{k}, & \mbox{if } \tau=2k,
     \end{array} .
\end{equation}
from the formula \eqref{defin_f0.1}. Moreover, we define $\tau(n,M)$ to be the unique positive integer $\tau$ such that
\[ M \in \left( D(n,\tau),D(n,\tau+1)\right). \]

The formula  \eqref{defin_f0.1} with $L_\tau(n,s)=M$ was called $1/M$-quadrature formula in \cite{BDHSS-codes}.

\subsection{Universal lower bound on energy of designs}

In \cite{BDHSS-codes} a lower bound on the energy of spherical codes was proved. This bound is
universal in the sense of Levenshtein \cite{Lev}. We present here its formulation for spherical designs \cite{BDHSS-designs}. Denote
\begin{equation}
\label{LE}
\mathcal{L}(n,M,\tau;h):=\inf \{E(n,C;h):|C|=M, \ C \subset \mathbb{S}^{n-1} \mbox{ is a $\tau$-design}\}.
\end{equation}

\begin{theorem}
Let $n \geq 3$, $\tau$, and $M \in [D(n,\tau), D(n,\tau+1))$ be positive integers.
Let $h:[-1,1]\to[0,+\infty]$ be absolutely monotone. Then
\begin{equation}
\label{ULB-designs}
\mathcal{L}(n,M,\tau;h) \geq \left\{
     \begin{array}{ll}
       M^2\sum_{i=0}^{k-1} \rho_i h(\alpha_i), & \mbox{if } \tau=2k-1, \\[6pt]
       M^2\sum_{i=0}^{k} \gamma_i h(\beta_i), & \mbox{if } \tau=2k
     \end{array} \right. .
\end{equation}
\end{theorem}

The main result in this paper shows that the strip between the lower bound \eqref{ULB-designs} and our upper bound
is very thin provided that the cardinality of the designs under consideration is relatively small; i.e., close to the Delsarte-Goethals-Seidel bound
\eqref{DGS-bound}.

\subsection{Restrictions on the structure of spherical designs}

Denote
\begin{equation}
\label{unNtau} u(n,M,\tau):=\sup \{u(C):C \subset \mathbb{S}^{n-1} \mbox{ is a $\tau$-design}, |C|=M\},
\end{equation}
where $u(C):=\max \{\langle x,y \rangle : x,y \in C, x \neq y\}$, and
\begin{equation}
\label{lnNtau}\ell(n,M,\tau):=\inf \{\ell(C):C \subset \mathbb{S}^{n-1} \mbox{ is a $\tau$-design}, |C|=M\},
\end{equation}
where $\ell(C):=\min \{\langle x,y \rangle : x,y \in C, x \neq y\}$.

For every $n$, $\tau$, and $M \in (D(n,\tau), D(n, \tau+1))$ non-trivial bounds on $u(n,M,\tau)$ are possible \cite{BBKS,BBD}.

\begin{lemma}
\label{sep_lemma1} {\rm \cite{BBD}} We have $u(n,M,2k-1) \geq \alpha_{k-1}$ and $u(n,M,2k) \geq \beta_k$.
\end{lemma}

\begin{lemma}
\label{sep_lemma0}
We have  $u(n,M,2k-1) > t_{k-1,k-1}^{1,1}$ and $u(n,M,2k) > t_{k,k}^{1,0}$.
\end{lemma}

\begin{proof}
This follows by Lemma \ref{sep_lemma1} and the inequalities $\alpha_{k-1}> t_{k-1,k-1}^{1,1}$ and $\beta_k > t_{k,k}^{1,0}$ from \cite[Theorem 5.39]{Lev}.
\end{proof}

We also utilize upper bounds on $u(n,M,\tau)$ from \cite{BBKS,BBD}.
Numerical examples can be found in \cite{MS-dissertation}. Here we list explicit results for $\tau=2$ and 4.

\begin{lemma}
\label{sep_lemma}
{\rm \cite{BBKS} a)} For every $n \geq 3$ and every $M \in [D(n,2),D(n,3)]=[n+1,2n]$ we have
\[ u(n,M,2) \leq \frac{M-2}{n}-1. \]

{\rm b)} For every $n \geq 3$ and every $M \in [D(n,4),D(n,5)]=[n(n+3)/2,n(n+1)]$ we have
\[ u(n,M,4) \leq \frac{2(3+\sqrt{(n-1)[(n+2)M-3(n+3)]})}{n(n+2)}-1. \]
\end{lemma}

The bounds from Lemma \ref{sep_lemma} are good when $M$ is close to $D(n,2)$ and $D(n,4)$, respectively,
and become worse with the increasing of $M$.


\subsection{Linear programming for upper energy bounds}

The next theorem, proven in \cite{BDHSS-designs} gives general technique for obtaining upper bounds for the energy of spherical designs of
fixed dimension, strength, and cardinality.

\begin{theorem}
\label{general-upper}
{\rm \cite{BDHSS-designs}} Let $n$, $\tau$, and $M \geq D(n,\tau)$ be positive integers. Let $h:[-1,1]\to[0,+\infty]$. Suppose that
$I$ is a subset of $[-1,1)$ and $g(t) = \sum_{i=0}^{\deg(g)} g_i P_i^{(n)}(t)$ is a real polynomial such that:

{\rm (D1)} $g(t) \geq h(t)$ for  $ t \in I$;

{\rm (D2)} the Gegenbauer coefficients of $g(t)$ satisfy $g_i \leq 0$ for $i \geq \tau+1$.

If $C \subset \mathbb{S}^{n-1}$ is a spherical $\tau$-design of $|C|=M$ points such that $\langle x,y\rangle \in I$ for distinct points
$x,y\in C$, then $E(n,C;h) \leq M(g_0M-g(1))$.
In particular, if $[\ell(n,M,\tau),u(n,M,\tau)] \subseteq I$, then
\begin{equation}\label{mainUbnd}
\mathcal{U}(n,M,\tau;h) \leq M(g_0M-g(1)).
\end{equation}
\end{theorem}

It is unknown which are the best polynomials for Theorem \ref{general-upper} even if their degree is restricted in advance. Our propositions,
as shown below, give upper energy bounds which, despite not optimal, are very close to the lower bounds from \cite{BDHSS-designs,BDHSS-codes}.

\subsection{Hermite interpolation}

According to condition (D1), good polynomials for Theorem \ref{general-upper} have to stay above the potential $h$.
This naturally leads to use of Hermite interpolation
which provides polynomial, whose graph is tangential to the graphs of $f$ and $h$. Thus we need interpolation that gives a polynomial matching
the values of $h$ and its derivative $h^\prime$ at certain points (to be specified later).

More precisely, we are given $m+1$ distinct points $t_1<t_2<\cdots<t_m<u$ in $[-1,1)$ and we wish to find a polynomial $f$ of degree less
than $2m+1$ (or less than $2m$ if $t_1=-1$) such that
\[ f(t_i)=h(t_i), \ i=1,2,\ldots,m, \ f(u)=h(u) \]
and
\[ f^\prime(t_1)=h^\prime(t_1)  \iff t_1>-1, \ f^\prime(t_i)=h^\prime(t_i), \ i=2,\ldots,m. \]
There always exists a unique such polynomial \cite{D}.
The next assertion concerning the interpolation error is well known \cite[Theorem 3.5.1]{D}.

\begin{lemma}
\label{rolle}
Under the hypotheses for the Hermite interpolation as explained above,
for every $t \in [-1,1]$ there exists $\xi \in (\min(t,t_1), \max(t,u))$ such that
\[  h(t) - f(t) =
\left\{ \begin{array}{ll}
        \frac{h^{(2m+1)}(\xi)}{((2m+1)!}(t-t_1)^{2}(t-t_2)^{2}\ldots(t-t_{m})^{2}(t-u), & \mbox{if } t_1>-1, \\[6pt]
       \frac{h^{(2m)}(\xi)}{(2m)!}(t+1)(t-t_2)^{2}\ldots(t-t_{m})^{2}(t-u), & \mbox{if } t_1=-1
     \end{array} . \right.
\]
\end{lemma}

We use Lemma \ref{rolle} in the proof that our polynomials satisfy (D1). The numbers $t_i$ will be the zeros
$t_{m,i}^{a,b}$ of the adjacent polynomials $P_m^{a,b}(t)$, where $m=k-1$ or $k$.

\section{Upper bounds for $\mathcal{U}(n,M,\tau;h)$}

\subsection{Derivation of the bounds}

We propose the usage of the roots of the polynomials \eqref{DGS-poly} which were used by Delsarte, Goethals and
Seidel \cite{DGS} for obtaining the bound \eqref{DGS-bound}. Our idea for this choice is motivated by the combination of two results.

First, there is, in some sense, duality between
lower bounds (by Delsarte, Goethals and Seidel) for the size of spherical designs of fixed dimension and strength and
upper bounds on the size of spherical codes (by Levenshtein)
of fixed dimension and minimum distance (see \cite{DL}). This is well illustrated by \eqref{L-DGS1}.

Second, the universal lower bound \eqref{ULB-designs} on the energy of spherical codes and designs was obtained by using interpolation in the nodes defined by Levenshtein for obtaining his
upper bounds on maximal codes. Therefore, we find it natural to use the nodes of the Delsarte-Goethals-Seidel's polynomials for
obtaining upper bounds on energy of designs.

In other words -- as Boyvalenkov-Dragnev-Hardin-Saff-Stoyanova \cite{BDHSS-designs,BDHSS-codes} used for their lower bounds interpolation in the nodes, coming from the Levenshtein polynomials, we decide to use for our upper bounds interpolation in the nodes coming from the Delsarte-Goethals-Seidel polynomials. We explain our interpolation scheme in detail in the proof of the next theorem.

\begin{theorem}
\label{Upper-bounds-thm}
Let $n \geq 3$, $\tau$, and $M \in (D(n,\tau), D(n,\tau+1))$ be positive integers.
Let $h:[-1,1]\to[0,+\infty]$ be absolutely monotone. Then
the polynomials constructed as in Cases 1 and 2 below satisfy the conditions (D1) and (D2) of Theorem \ref{general-upper}
and produce upper bounds for the corresponding $\mathcal{U}(n,M,\tau;h)$.
\end{theorem}

\begin{proof} We first explain the ends of our intervals $I$ for $f$ exceeding $h$ as Theorem $\ref{general-upper}$
requires. We always set $I=[-1,u]$, where $u \geq u(n,M,\tau)$ is valid upper bound (as in Lemma \ref{sep_lemma} for $\tau=2$ and 4 or numerical).
Therefore we have $[\ell(n,M;\tau),u(n,M;\tau)] \subseteq I$.

Next, we interpolate in the roots of the Delsarte-Goethals-Seidel polynomials $d_{\tau}(t)$ and the point $u$ as follows:

\paragraph{Case 1.} For $\tau=2k-1$ we choose $g(t)$ as the Hermite interpolant of $h$ at the roots $-1$ and
$t_{k-1,i}^{1,1}$, $i=1,2,\ldots,k-1$, of the polynomial $d_{2k-1}(t)$, and at $u$ as follows
\[ g(-1)=h(-1), \ g(t_{k-1,i}^{1,1})=h(t_{k-1,i}^{1,1}), \ g^\prime(t_{k-1,i}^{1,1})=h^\prime(t_{k-1,i}^{1,1}), \ i=1,2,\ldots,k-1,\]
\[g(u)=h(u). \]
These are in total $1+2(k-1)+1=2k$ conditions. Then $g(t)$ is a polynomial of degree at most $2k-1$
which satisfies by Lemmas \ref{sep_lemma0} and \ref{rolle}
the condition (D1) of Theorem \ref{general-upper}. Indeed, we have
\[ h(t) - g(t) = \frac{h^{(2k)}(\xi)}{(2k)!}(t+1)(t-t_{k-1,1}^{1,1})^2 \ldots (t-t_{k-1,k-1}^{1,1})^2(t-u) \leq 0 \]
for every $t \in I=[-1,u]$ (note that $h^{(2k)}(\xi) \geq 0$ since $h$ is absolutely monotone).

The condition (D2) is trivially satisfied since $\deg(g) \leq \tau= 2k-1$ and
therefore we have an upper bound for $\mathcal{U}(n,M,2k-1;h)$ produced by our polynomial.

\smallskip

\paragraph{Case 2.} For $\tau=2k$ we choose $g(t)$ as the Hermite interpolant of $h$ at the roots
$t_{k,i}^{1,0}$, $i=1,2,\ldots,k$, of $d_{2k}(t)$, and at $u$ as follows
\[ g(t_{k,i}^{1,0})=h(t_{k,i}^{1,0}), \ g^\prime(t_{k,i}^{1,0})=h^\prime(t_{k,i}^{1,0}), \ i=1,2,\ldots,k,
g(u)=h(u). \]
Now $g(t)$ is a polynomial of degree $2k$. Since
\[ h(t) - g(t) = \frac{h^{(2k+1)}(\xi)}{(2k+1)!}(t-t_{k,1}^{1,0})^2 \ldots (t-t_{k,k}^{1,0})^2(t-u) \leq 0 \]
for every $t \in I= [-1,u]$ by Lemma \ref{rolle} and due to the absolute monotonicity of $h$,
the condition (D1) of Theorem \ref{general-upper} is satisfied.

The condition (D2) is again trivially satisfied and
therefore we have an upper bound for $\mathcal{U}(n,M,2k;h)$.
\end{proof}

\subsection{Two representations of the new upper bounds}

Similarly to the ULB, our bound from Theorem \ref{Upper-bounds-thm} can be written to include certain values of the potential function $h(t)$. Again this is done by using the Levenshtein's quadrature rule.

\begin{theorem}
\label{Upper-bounds-thm-2}
Let $n \geq 3$, $\tau$, and $M \in (D(n,\tau), D(n,\tau+1))$ be positive integers.
Let $h:[-1,1]\to[0,+\infty]$ be absolutely monotone. Then
\begin{equation}
\label{UB-1-designs}
\frac{\mathcal{U}(n,M,\tau;h)}{M^2} \leq \left\{
     \begin{array}{ll}
    \ds  g_0\left(1-\frac{D(n,2k-1)}{M}\right)+\frac{D(n,2k-1)}{M}
                \sum_{i=0}^{k-1} \gamma_i h(t_{k-1,i}^{1,1}) , \\ & \mbox{if } \tau=2k-1 \\[12pt]
   \ds  g_0\left(1-\frac{D(n,2k)}{M}\right)+\frac{D(n,2k)}{M}
                \sum_{i=0}^{k-1} \rho_i h(t_{k,i+1}^{1,0}), & \mbox{if } \tau=2k
     \end{array} \right.
\end{equation}
(set $t_{k-1,0}^{1,1}:=-1$ in the odd case).
\end{theorem}

\begin{proof} We explain in more detail the even case $\tau=2k$. We start with the
$1/D(n,2k)$-quadrature rule. Solving $L_{2k-1}(n,t_{k}^{1,0})=D(n,2k)$ (see the second line in \eqref{L-DGS1}) to produce the parameters $\rho_i$ and $\alpha_i=t_{k,i+1}^{1,0}$, $i=0,1,\ldots,k-1$, we obtain
\[ g_0D(n,2k)-g(1) = D(n,2k)\sum_{i=0}^{k-1} \rho_i g(t_{k,i+1}^{1,0})=D(n,2k)\sum_{i=0}^{k-1} \rho_i h(t_{k,i+1}^{1,0}). \]
(the interpolation equalities $g(t_{k,i+1}^{1,0})= h(t_{k,i+1}^{1,0})$ were used). Therefore,
\begin{eqnarray*}
\mathcal{U}(n,M,\tau,h) &\leq& M(g_0M-g(1)) \\
&=& M(g_0M(M- D(n,2k))+(g_0D(n,2k)-g(1)) \\
&=& M(g_0(M-D(n,2k))+D(n,2k)\sum_{i=0}^{k-1} \rho_i h(t_{k,i+1}^{1,0})),
\end{eqnarray*}
which completes the proof in the even case. The odd case  $\tau=2k-1$ is analogous.
\end{proof}

\begin{theorem}
\label{Upper-bounds-thm-3}
Let $n \geq 3$, $\tau$, and $M \in (D(n,\tau), D(n,\tau+1))$ be positive integers.
Let $h:[-1,1]\to[0,+\infty]$ be absolutely monotone. Then
\begin{equation}
\label{UB-2-designs}
\frac{\mathcal{U}(n,M,\tau;h)}{M^2} \leq \left\{
     \begin{array}{ll}
    \ds  \frac{ULB}{M^2}+\sum_{i=0}^{k-1} \rho_i (g(\alpha_i)-h(\alpha_i)), & \mbox{if } \tau=2k-1 \\[6pt]
   \ds \frac{ULB}{M^2}+\sum_{i=0}^k \gamma_i (g(\beta_i)-h(\beta_i)), & \mbox{if } \tau=2k,
     \end{array} \right.
\end{equation}
where ULB is the bound \eqref{ULB-designs} for the corresponding branch, and the parameters $\rho_i$, $\alpha_i$,
$\gamma_i$, and $\beta_i$ are the same as in \eqref{ULB-designs}.
\end{theorem}

\begin{proof} In the even case $\tau=2k$ we solve $M=L_{2k}(n,s)$ and derive
the parameters $\gamma_i$ and $\beta_i$, $i=0,1,\ldots,k$. Using the $1/M$-quadrature rule we have
\[ g_0M-g(1)=M\sum_{i=0}^k \gamma_i g(\beta_i). \]
Therefore our upper bound can be written as
\begin{eqnarray*}
\mathcal{U}(n,M,\tau,h) &\leq& M(g_0M-g(1)) = M^2 \sum_{i=0}^k \gamma_i g(\beta_i) \\
&=& ULB+M^2 \sum_{i=0}^k \gamma_i (g(\beta_i)-h(\beta_i)).
\end{eqnarray*}
The odd case is analogous. \end{proof}

\begin{remark} Using the remainder formula from Lemma \ref{rolle} one can write, for example,
\begin{eqnarray*}
 g(\beta_i)-h(\beta_i) &=& \frac{h^{(2k)}(\xi_i)}{(2k)!} (\beta_i-t_{k,1}^{1,0})^2 \ldots (\beta_i-t_{k,k}^{1,0})^2(u-\beta_i) \\
 &=& \frac{h^{(2k)}(\xi_i)}{(2k)!} (P_k^{1,0}(\beta_i))^2(u-\beta_i),
\end{eqnarray*}
where $\xi_i$, $i=0,1,\ldots,k$, are points from the interval $\mathcal{I}=[-1,u]$. This
estimate can be further investigated. The odd case $\tau=2k-1$ is similar.
\end{remark}

\section{Comparisons and examples}

\subsection{Explicit bounds for 2-designs}

We derive explicit bounds for $\tau=2$ and compare them to the lower and upper bounds from \cite{BDHSS-designs}.

\begin{theorem}
\label{up-2}
For given $n$, $M \in (D(n,2),D(n,3))=(n+1,2n)$, and absolutely
monotone\footnote{Clearly, it is enough to have positive derivatives only the first three derivatives,} in $[-1,1)$ function $h$ we have
\begin{equation}
\label{up2explicit}
\frac{\mathcal{U}(n,M,2;h)}{M^2} \leq \frac{M-1}{M}\left(h\left(-\frac{1}{n}\right)+\frac{1}{n}h'\left(-\frac{1}{n}\right)+\frac{A(n+1)}{n^2}\right)
-\frac{nh'\left(-\frac{1}{n}\right)+A(n+1)}{nM} ,
\end{equation}
where $\displaystyle
A=\frac{h\left(\frac{M-2}{n}-1\right)-h\left(-\frac{1}{n}\right)-\left(\frac{M-1}{n}-1\right)h'\left(-\frac{1}{n}\right)}{\left(\frac{M-1}{n}-1\right)^2}.
$
\end{theorem}

\begin{proof}
According to Theorem \ref{Upper-bounds-thm} we have to double interpolation nodes in $t_{1,1}^{1,0}=-1/n$ and singular in $u=(M-2)/n-1$.
Straightforward calculations by the Newton formula give the desired bound.
\end{proof}

Theorem \ref{up-2} and Theorem 4.2 from \cite{BDHSS-designs} determine thin asymptotic strip for the energies of spherical 2-designs with
cardinalities in $(n+1,2n)$.

\begin{theorem}
\label{asy-2}
If $n$ and $M$ tend to infinity in relation $M=\lambda n + o(1)$, where $\lambda \in (1,2)$ is a constant, then
\begin{eqnarray*}
h(0)+\frac{h(1-\lambda)-\lambda h(0)}{M(\lambda-1)} &\leq& \frac{\mathcal{L}(n,M,2;h)}{M^2} \\
&\leq& \frac{\mathcal{U}(n,M,2;h)}{M^2} \leq h(0)+\frac{h(\lambda-1)-h(0)-2(\lambda-1)h'(0)}{M(\lambda-1)}.
\end{eqnarray*}
\end{theorem}

\begin{proof}
The lower bound is given by (32) in Theorem 4.2 from \cite{BDHSS-designs}. The upper bound easily follows from \eqref{up2explicit} in
Theorem \ref{up-2}.
\end{proof}

On Figure 1 we exhibit the situation for $n=20$, $\tau=2$, $22 \leq M \leq 28$, and $h(t)=[2(1-t)]^{-(n-2)/2}$ -- the Newton potential.
Our bound \eqref{up2explicit} is $U_1$, $L$ and $U_2$ are the lower bound (31) and the upper bound (36), respectively,
from \cite{BDHSS-designs}. It is worth to note that $L=U_2$ for $M=n+2$ in every dimension (see Example 5.1 in \cite{BDHSS-designs}) and this seems to be the only case where our bound is weaker.

\begin{figure}[ht]\label{NE-figure}
\includegraphics[scale=0.75]{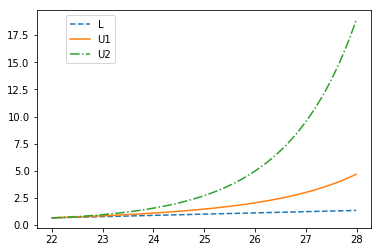}
\vskip 5mm
\caption{Newtonian (harmonic) energy comparison for $n=10$, $22  \leq M \leq 28$.}
\label{fig1}
\end{figure}

\subsection{Numerical examples}

We present the typical situation by giving numerical examples of bounds for spherical 4- and 5-designs. 

In \cite{BDHSS-designs} the interpolation rules $g(\ell)=h(\ell)$, $g(a_{i})=h(a_{i})$, $g'(a_{i})=h'(a_{i})$, and $g(u)=h(u)$ were
applied for $\tau \leq 4$ with polynomials of degree 1 (for $\tau \leq 2$) and 3 (for $\tau = 3$ and 4).
Here $u$ and $\ell$ are suitable upper and lower bounds for $u(n,M,\tau)$ and $\ell(n,M,\tau)$, respectively, $i=\lfloor (\tau-1)/2 \rfloor$,
and $a_{i}$ are suitably chosen. The resulting bounds are optimal for that approach but worse than our bounds.

In Table 1 we present numerical examples with bounds for spherical 4-designs in dimensions $3 \leq n \leq 10$.
The feasible cardinalities are
\[ M \in (D(n,4),D(n,5))=\left(\frac{n(n+3)}{2}+1,n^2+n\right), \]
and the potential function is again the Newtonian $h(t)=[2(1-t)]^{-(n-2)/2}$. Similarly to above, we show our bound as $U_1$, the
universal lower bound $L$ \cite[Theorem 3.4]{BDHSS-designs}, and the upper bound $U_2$ from Theorem 5.2 in \cite{BDHSS-designs}:
\begin{eqnarray*}
\label{ub_deg3}
\hspace*{4mm} \mathcal{U}(n,N,4;h) &\leq& N\big((N-1)h(a_0) \\
&& +\,\frac{(h(\ell)-h(a_0))\left[uN(1+na_0^2)+2Na_0+n(1-u)(1-a_0)^2\right]}{n(u-\ell)(\ell-a_0)^2} \nonumber \\
&& -\,\frac{(h(u)-h(a_0))\left[\ell N(1+na_0^2)+2Na_0+n(1-\ell)(1-a_0)^2\right]}{n(u-\ell)(u-a_0)^2}\big), \nonumber
\end{eqnarray*}
  where
$\displaystyle a_0:=\frac{N(\ell+u)+n(1-\ell)(1-u)}{n(1-\ell)(1-u)-N(1+\ell un)}$,
$\displaystyle \ell:=1-\frac{2}{n}\left(1+\sqrt{\frac{(n-1)(N-2)}{n+2}}\right)$ is
a lower bound for $\ell(n,4,M)$ as in Lemma 2.2 from \cite{BDHSS-designs}, and $u$ is as in Lemma \ref{sep_lemma}.

\begin{center}
\begin{tabular}{|c|c|c|c|c|}
\hline
$n$  & $M$ & $U_2$ & $U_1$ & $L$ \\
\hline
 3 & 10 & 65.81 & 65.57 & 65.34 \\
 \hline
 3 & 11 & 81.99 & 81.52 & 80.98 \\
 \hline
 4 & 15 & 117.62 & 115.62 & 114.95 \\
 \hline
 4 & 16 & 137.72 & 134.86 & 133.33 \\
 \hline
 4 & 17 & 160.16 & 155.83& 153.125 \\
 \hline
 4 & 18 & 185.32 & 178.66 & 174.33 \\
 \hline
 4 & 19 & 213.79 & 203.55 & 196.95 \\
 \hline
 5 & 21 & 183.89 & 176.78 & 175.50 \\
 \hline
  5 & 22 & 207.71 & 198.50 & 195.63 \\
 \hline
  5 & 23 & 233.97 & 221.78 & 216.92 \\
 \hline
  5 & 24 & 263.07 & 246.76 & 239.35 \\
 \hline
  5 & 25 & 295.54 & 273.59 & 262.95 \\
 \hline
  5 & 26 & 332.06 & 302.50 & 287.69 \\
 \hline
  5 & 27 & 373.55 & 333.77 & 313.59 \\
 \hline
  5 & 28 & 421.33 & 367.80 & 340.65 \\
 \hline
  5 & 29 & 477.22 & 405.17 & 368.86  \\
 \hline
    \end{tabular} \\[6pt]
{\bf Table 1.} Newtonian energy comparison for $n=3,4,5$, $\tau=4$, $1+n(n+3)/2  \leq M \leq n(n+1)-1$.   \end{center}

To our knowledge there are no upper bounds for the energy of 5-designs in the literature to compare our results with. Thus we show in Table 2 our bound $U_1$ and the universal lower bound $L$ from \cite[Theorem 3.4]{BDHSS-designs}. We note that 5-designs with $n^2+n+1$ points in $n$ dimensions do not exist \cite{BDN1999}. 
The values of $u$ are taken from \cite{MS-dissertation}.
   
\begin{center}
\begin{tabular}{|c|c|c|c|c|}
\hline
$n$  & $M$ & $U_1$ & $L$ \\
\hline
 3 & 13 & 117.77 & 117.50 \\
 \hline
 3 & 14 & 139.04 & 138.43 \\
 \hline
 3 & 15 & 162.18 & 161.12 \\
 \hline
 3 & 16 & 134.86 & 185.56 \\
 \hline
 4 & 21 & 247.63 & 246.75 \\
 \hline
 4 & 22 & 275.92 & 274 \\
 \hline
 4 & 23 & 305.92 & 302.75 \\
 \hline
 4 & 24 & 337.71 & 333 \\
 \hline
 5 & 31 & 431.12 & 429.26 \\
 \hline
 5 & 32 & 465.53 & 461.55 \\
 \hline
 5 & 33 & 501.52 & 495.10 \\
 \hline
 5 & 34 & 539.15 & 529.90 \\
 \hline
    \end{tabular} \\[6pt]
{\bf Table 2.} Newtonian energy comparison for $n=3,4,5$, $\tau=5$, $n(n+1)+1  \leq M \leq n(n+1)+4$.   \end{center}

Finally, it is worth to note that our bounds can be slightly improved by more flexible choice of the interpolation nodes.
Indeed, one can apply a numerical method to move the interpolation nodes like in \cite{Boy95}, where such idea was used for obtaining
linear programming bounds for spherical codes and designs.

{\bf Acknowledgement.} Peter Boyvalenkov is also with Faculty of Mathematics and Natural Sciences, South-Western University, Blagoevgrad, Bulgaria.

\end{document}